\documentclass[a4paper,12pt,reqno]{amsart}
\usepackage{amsmath,amsfonts,enumerate,amssymb,amsthm,color}
\usepackage{pgf,tikz}

\title{Integral automorphisms of affine spaces over finite fields}
\author{Istv\'an~Kov\'acs}
\address{I.~Kov\'acs, 
IAM and FAMNIT, University of Primorska, Glagolja\v ska 8, 6000 Koper, Slovenia}
\email{istvan.kovacs@upr.si}
\author{Klavdija Kutnar}
\address{K.~Kutnar, 
IAM and FAMNIT, University of Primorska, Glagolja\v ska 8, 6000 Koper, Slovenia}
\email{klavdija.kutnar@upr.si}
\author{J\'anos Ruff}
\address{J.~Ruff, 
Institute of Mathematics and Informatics, 
University of P\'ecs, Ifj\'us\'ag \'utja 6, 7624 P\'ecs, Hungary}
\email{ruffjanos@gmail.com}
\author{Tam\'as Sz\H onyi}
\address{T.~Sz\H onyi, 
Institute of Mathematics, E\"otv\"os University 
\& MTA-ELTE Geometric and Algebraic Combinatorics Research Group,
P\'azm\'any P. s. 1/C, 1117 Budapest, Hungary }
\email{szonyi@cs.elte.hu}


\newtheorem{thm}{Theorem}
\newtheorem{lem}{Lemma}

\theoremstyle{remark}
\newtheorem{rem}{Remark}

\def\F{\mathbb{F}}
\def\R{\mathbb{R}}
\def\Q{\mathcal{Q}}

\def\va{\mathbf{a}}
\def\vb{\mathbf{b}}

\def\vx{\mathbf{x}}
\def\vy{\mathbf{y}}
\def\vz{\mathbf{0}}
\DeclareMathOperator{\ag}{AG}

\DeclareMathOperator{\aut}{Aut}
\DeclareMathOperator{\gl}{GL}
\DeclareMathOperator{\ggl}{\Gamma L}
\DeclareMathOperator{\pg}{PG}
\DeclareMathOperator{\soc}{Soc}
\usepackage{color}

\begin{document}

\maketitle

\begin{abstract} 
A permutation of the point set of the affine space $\ag(n,q)$ is called an integral automorphism if it preserves the integral distance defined among the points. In this paper, we complete the classification 
of the integral automorphisms of $\ag(n,q)$ for $n\ge 3$.
\end{abstract}

\section{Introduction}
Throughout the paper $p$ stands for an odd prime.  
Let $\F_q$ be the finite field with $q=p^h$ elements  and $\ag(n,q)$ 
be the $n$-dimensional affine space defined over $\F_q$. The \emph{Euclidean distance} $d$ is defined as 
$$
d(\vx,\vy)=\sum_{i=1}^{n}(x_i-y_i)^2
$$
for the points $\vx=(x_1,\dots,x_n)$ and 
$\vy=(y_1,\dots,y_n) \in \F_q^n$. Two points $\vx$ and $\vy$ are said to be at \emph{integral distance} if $d(\vx,\vy)$ is a square element 
in $\F_q$ including the zero element,
and a set of points is called \emph{integral} if any two of its points are at integral distance. Recently, the finite field analog of the classical problem about integral point sets in $\R^n$ has attracted considerable attention. See, for example, \cite{ISX} and the references therein. Besides integral point sets, permutations, preserving integral distances, are also considered in \cite{KK,KR,K,KM}. 
By an \emph{integral automorphism} of $\ag(n,q)$ we mean any bijective mapping $\gamma : \F_q^n \to \F_q^n$ satisfying 
$$
d(\vx,\vy) \in   \square_q \iff d(\vx^\gamma,\vy^\gamma) \in \square_q
$$ 
for all $\vx,\,\vy \in \F_q^n$. 
Here and in what follows $\square_q$ denotes the set of all square elements of $\F_q$ including the zero element. 
We adopt the notation used in \cite{KK} and denote 
the group of all integral automorphisms by $\aut(\F_q^n)$.

Integral automorphisms of the plane $\ag(2,q)$ 
were determined in \cite{KK,KR,K}. In particular,  
$\aut(\F_q^2)$ was found by Kurz \cite{K} for 
$q\equiv 3\pmod 4,$ and by Kov\'acs and Ruff \cite{KR}  
for $q\equiv 1\pmod 4$. We remark that the special case  
$q=p$ was settled earlier by Kiermaier and Kurz \cite{KK}. 
It turns out that there exist integral automorphisms of $\ag(2,q)$
which are not semiaffine transformations, and this occurs exactly 
when $q\equiv 1\pmod 4$. As for higher dimensions, Kurz and Meyer \cite{KM} described the integral automorphisms which are also semiaffine transformations. In what follows we denote by $\F_q^\times$ the multiplicative group of $\F_q,$ by $\gl(n,q)$ the group of  invertible $n \times n$ matrices with entries from $\F_q,$ and by 
$\sigma$ the semiaffine transformation defined by $(x_1\ldots,x_n) \mapsto (x_1^p,\ldots,x_n^p)$.

\begin{thm}[Kurz and Meyer \cite{KM}]\label{KM} 
If $q=p^h$ and $n \ge 3,$ then the semiaffine transformations contained in $\aut(\F_q^n)$ are given as 
$$
\vx \mapsto a \vx^{\sigma^i}A+\vb
$$
where $a \in \F_q^\times, i \in \{0,\ldots,h-1\}, A \in \gl(n,q)$ 
with $AA^T=I$ and $\vb \in \F_q^n$. 
\end{thm}

Our goal in this paper is to show that, in contrast with the plane, 
all integral automorphisms of $\ag(n,q)$ are semiaffine 
transformations whenever $n\ge 3$. 
This together with Theorem~\ref{KM} result in the following 
classification theorem.

\begin{thm}\label{T-main}
Let $q=p^h$ for an odd prime $p$ and suppose that $n\ge 3$. 
Then the integral automorphisms of $\ag(n,q)$ are the mappings 
$$
\vx \mapsto a \vx^{\sigma^i}A+\vb
$$
where $a \in \F_q^\times, i \in \{0,\ldots,h-1\}, A \in \gl(n,q)$ 
with $AA^T=I$ and $\vb \in \F_q^n$. 
\end{thm}

\section{Proof of Theorem~\ref{T-main}}
The key part in the proof of Theorem~2 will be to show that every integral automorphism $\gamma \in \aut(\F_q^n)$ satisfies 
\begin{equation}\label{iff}
d(\vx,\vy)=0 \, \iff \, d(\vx^\gamma,\vy^\gamma)=0 \text{ for all } \vx,\, \vy \in \F_q^n.
\end{equation}
This enables us to use 
the result of Lester~\cite{Les} about cone preserving mappings. 
Let $V$ be a nonsingular metric vector space over a field $\F$ whose  characteristic is different from two, upon which is defined a nonsingular symmetric bilinear form $\langle .,.\rangle$.  The \emph{cone $C(\va)$ with vertex} $\va \in V$ is defined to be the set 
$C(\va):=\{ \vx \in V : \langle\vx-\va,\vx-\va\rangle=0 \},$ and a  
mapping $f : V \to V$ is said to \emph{preserve cones} if 
$(C(\va))^f=C(\va^f)$ for all $\va\in V$. 

\begin{thm}[Lester \cite{Les}]\label{L}
Let $V$ be a nonsingular metric vector space over the field $\F$, with bilinear form $\langle.,.\rangle$; assume that $\dim(V) \ge 3$ and that
$V$ is not anisotropic (that is, $\langle \vx,\vx \rangle = 0$ for 
some nonzero vector $\vx$). Let $f : V \to V$ be a bijection of $V$ which preserves cones. Then $f$ is of the form
$$
f : \vx \mapsto L(\vx)+\vb 
$$
where $\vb \in V,$ and $(L,\rho)$ is a semilinear transformation of 
$V$ satisfying $\langle L(\vx),L(\vy) \rangle=a \langle \vx,\vy\rangle^\rho$ for some nonzero $a\in \F$ and for all $\vx,\,\vy\in V$. 
\end{thm}

Now, if $\gamma \in \aut(\Gamma)$ satisfies \eqref{iff}, then 
it preserves cones of the metric vector space 
$V:=\F_q^n$ equipped with the symmetric bilinear form 
$\langle .,. \rangle$ defined by $\langle\vx,\vy\rangle:=\vx \vy^T$ for all vectors $\vx,\,\vy\in V$. Therefore, by 
Theorem~\ref{L}, $\gamma$ is a semiaffine 
transformation, and Theorem~\ref{T-main} follows. 
In fact, we are going to derive \eqref{iff} in the end of this section following two preparatory lemmas.
\medskip

For the rest of the paper we let $G=\aut(\F_q^n),$ $n\ge 3,$ and let $G_\vz$ be  the stabilizer of $\vz$ in $G$ where $\vz=(0,\ldots,0) \in \F_q^n$.  We start by introducing two subgroups of $G$:
\begin{eqnarray*}
E &=& \big\{ \vx \mapsto \vx + \vb : \vb \in \F_q^n \big\}, \\
M &=& \big\{ \vx \mapsto a \vx A : a \in \F_q^\times,  
A \in \gl(n,q)\text{ and }AA^T=I \big\}.
\end{eqnarray*}

Notice that, by Theorem~\ref{KM}, both $E$ and $M$ are subgroups 
of $G$.  The elements of $E$ are also called \emph{translations}. 
Clearly, $E$ is an elementary abelian group of order $p^{hn},$ and it is regular on $\F_q^n$. The group $M$ normalizes $E,$ hence $\langle E,M \rangle = EM$.  

Define the subsets of $\F_q^n$ as  
\begin{eqnarray*}
S_0 &=&  \Big\{ \vx \in \ag(n,q) : \sum_{i=1}^nx_i^2=0, 
\vx \ne \vz \Big\}, \\
S_+ &=&  \Big\{ \vx \in \ag(n,q) : \sum_{i=1}^nx_i^2 \in \square_q \setminus \{0\} \Big\}, \\
S_- &=&  \Big\{ \vx \in \ag(n,q) : \sum_{i=1}^nx_i^2 \notin \square_q \Big\}.
\end{eqnarray*}

\begin{lem}\label{L1}
With the above notation,
\begin{enumerate}[(i)]
\item The $M$-orbits are $\{\vz\}, S_0, S_+$ and $S_-$.   
\item $EM$ is primitive on $\F_q^n$.
\end{enumerate}
\end{lem}
\begin{proof}
Part (i) is proved in \cite[Lemma~3.17]{KM}.

To settle (ii) we apply \cite[Theorem~3.2A]{DM}, that is, 
$EM$ is primitive if and only if Graph$(\Delta)$ is connected for each nondiagonal orbital $\Delta$ of $EM$.  
Observe that, a nondiagonal orbital $\Delta$ consists of the ordered pairs in the form $(\vx,\vx+\vy),$ where $\vx$ runs over $\F_q^n$ 
and $\vy$ runs over $S_\varepsilon$ for a fixed $\varepsilon \in 
\{0,+,-\}$. Now, the connectedness of Graph$(\Delta)$ follows because each of $S_0, S_+$ and $S_-$ spans the vector space $\F_q^n$.
\end{proof}

By Lemma~\ref{L1}(i), $EM$ has nontrivial subdegrees 
$|S_\varepsilon|, \varepsilon \in \{0,+,-\}$. The exact values were 
computed in \cite[Theorem~4.3]{KM}:
\begin{eqnarray}\label{carS}
|S_0|&=& 
\begin{cases} 
q^{n-1}-1 & \text{ if $n$ is odd} \\
q^{n-1}+(-1)^{\frac{\varepsilon n}{2}}q^{\frac{n}{2}}-(-1)^{\frac{\varepsilon n}{2}}q^{\frac{n-2}{2}}-1 & \text{ if $n$ is even} 
\end{cases}  \label{S0} \\
|S_+|&=&  
\begin{cases} 
\frac{1}{2}\Big(q^n-q^{n-1}+(-1)^{\frac{\varepsilon(n+3)}{2}}
q^{\frac{n+1}{2}}-
(-1)^{\frac{\varepsilon(n-1)}{2}}
q^{\frac{n-1}{2}}\Big) & \text{ if $n$ is odd} \\
\frac{1}{2}\Big(q^n-q^{n-1}-(-1)^{\frac{\varepsilon n}{2}}
q^{\frac{n}{2}}+
(-1)^{\frac{\varepsilon n}{2}}
q^{\frac{n-2}{2}}\Big) & \text{ if $n$ is even} 
\end{cases} \label{S+} \\
|S_-|&=&
\begin{cases} 
\frac{1}{2}\Big(q^n-q^{n-1}-(-1)^{\frac{\varepsilon(n+3)}{2}}
q^{\frac{n+1}{2}}+
(-1)^{\frac{\varepsilon(n-1)}{2}}
q^{\frac{n-1}{2}}\Big) & \text{if $n$ is odd} \\ 
\frac{1}{2}\Big(q^n-q^{n-1}-(-1)^{\frac{\varepsilon n}{2}}
q^{\frac{n}{2}}+
(-1)^{\frac{\varepsilon n}{2}}
q^{\frac{n-2}{2}}\Big) & \text{if $n$ is even} \label{S-}
\end{cases}
\end{eqnarray}
where $\varepsilon=0$ if $q \equiv 1\pmod 4$ and 
$\varepsilon=1$ otherwise.

The set $S_0 \cup S_+$ consists of the points being at integral distance from $\vz$. Therefore, every $\gamma \in G_\vz$ 
maps $S_0 \cup S_+$ to itself, and this together with 
Lemma~\ref{L1}(i) leave us with two possibilities 
for the nontrivial $G_\vz$-orbits, namely, these are either     
$S_0,S_+$ and $S_-,$ or $S_0 \cup S_+$ and $S_-$.  In particular, 
the group $G$ has rank either $3$ with nontrivial subdegrees 
$|S_0|+|S_+|$ and $|S_-|,$ or $4$ with nontrivial subdegrees 
$|S_0|, |S_+|$ and $|S_-|$.

As the next step, we find the \emph{socle} $\soc(G)$. 
Recall that $\soc(G)$ is the subgroup of $G$ generated by 
all its minimal normal subgroups. 

\begin{lem}\label{L2}
With the above notation, the socle $\soc(G)=E$.
\end{lem}
\begin{proof}
Let $H=\soc(G)$. 
Since $EM\le G$ is primitive, see Lemma~\ref{L1}(ii), $G$ is 
primitive as well.  Thus $H$ is a direct product of isomorphic simple groups (see \cite[Corollary~4.3B]{DM}), and we may write 
$H=T \times \cdots \times T = T^k$ for some simple group  
$T$ and $k \ge 1$. By the O'Nan-Scott theorem, $G$ and $H$ are 
described by one of the following types 
(see, for example, \cite[pp.~137]{DM}): 
\begin{enumerate}[(T1)]
\item  $H$ is an elementary abelian $p$-group of order $q^n$ 
which is regular on $\F_q^n$.
\item $H$ is nonabelian and regular on $\F_q^n$. 
\item $H = T$  is nonabelian, it is not regular on $\F_q^n,$ 
and $G \le \aut(H)$.
\item $H$ is nonabelian and $G$ is a subgroup of a wreath product 
with the diagonal action. In this case $k \ge 2$ and $|T|^{k-1}=q^n$. 
\item $H$ is nonabelian, $k = k_1k_2$ and $k_2 > 1$. The group  
$G$ is isomorphic to a subgroup of the wreath product $U \, wr \, S_{k_2}$ with the product action, where $U$ is a primitive permutation 
group of degree $d$ such that $q^n=d^{k_2},$ $U$ has socle $T^{k_1},$ and $U$ is of type (T3) or (T4).
\end{enumerate}

We show below that $G$ is of type (T1). 
It is not hard to show that this yields $H=E$ (see, for example, \cite{KR}).  Now, suppose to the contrary that $G$ is described by one of types $(T2)-(T5)$. In each case $T$ is a nonabelian simple group. 
This observation excludes at once types (T2) and (T4).

Suppose next that $G$ is of type (T3). 
Then $T = H$, and since it is a normal subgroup of a primitive group, it acts transitively on $\F_q^n$.  It was proved 
by Guralnick \cite{G} that, if a finite nonabelian simple group $L$ acts transitively on a set $\Omega$ such that 
$|\Omega|$ is a  prime power, then $L$ acts $2$-transitively unless $L \cong \mathrm{PSU}(4,2)$ and 
$|\Omega| = 27$ with nontrivial subdegrees $10$ and $16$  (see \cite[Corollary~2]{G}). 
Since $G$ cannot be $2$-transitive, $q^n=27$ and the nontrivial subdegrees of $G$ are $10$ and $16$. This, however, contradicts 
that $|S_-|=12$ is a subdegree, see the remark before 
the lemma and \eqref{S-}. 

We are left with the case that $G$ is of type (T5). 
Denote by $r_G$ and $r_U$ the rank of $G$ and $U$, respectively.
Recall that $r_G \in \{3,4\}$.
By \cite[Exercise~4.8.1]{DM}, 
\begin{equation}\label{eq:bound}
r_G \ge { r_U + k_2 - 1 \choose k_2}.
\end{equation}

The group $U$ is of type (T3) or (T4). In the latter case  $|T| = p^a$ for some $a,$ a contradiction.  Thus $U$ is of type (T3), $k_1 = 1,$ 
$k = k_2$ and 
$T$ is a transitive permutation group of a set $X$ of size $|X|=q^{n/k_2}$. By the aforementioned result of Guralnick, 
$U$ is $2$-transitive unless $T \cong \mathrm{PSU}(4,2),$ 
$q^{n/k_2} = 27,$ and $r_U=3$. In the latter case, however,  we find in 
\eqref{eq:bound} that $r_G \ge \frac{1}{2}(k_2+2)(k_2+1) \ge 6$ (recall that 
$k_2 > 1$), a contradiction. 
 Thus $r_U = 2,$ implying  in \eqref{eq:bound} that $k=k_2=2$ and $r_G=3,$ or 
$k=k_2=3$ and $r_G=4$. 
\medskip

\noindent {\it Case 1.} $k_2 = 2,$ $r_G=3$ and $G \le U \, wr \, S_2$. 
\medskip

The wreath product $U \, wr \, S_2$ acts by the product action. 
This means that $\F_q^n$ can be written as $\F_q^n=X \times X,$ 
$|X|=q^{n/2},$ and $U$ is a permutation group of $X$. We have $U \, wr \, S_2 = \langle U \times U, \tau \rangle=
\langle U \times U \rangle \rtimes \langle \tau \rangle,$ where  
$U \times U$ acts on $X \times X$ naturally, and $\tau$ acts by switching the coordinates. The socle $H = T \times T \le U \times U,$  and since $T$ is $2$-transitive on $X,$ 
$\Delta_1:=\{ (x_0,x) : x \in X \setminus \{x_0\} \}$ and 
$\Delta_2:=\{ (x,x_0) : x \in  X \setminus \{x_0\} \}$ are 
orbits under the stabilizer $(U \times U)_{(x_0,x_0)},$ and 
any other orbit different from $\{ (x_0,x_0)\}$ is contained in the 
set $\Delta_3:=\{ (x,y) : x,y \in  X \setminus \{x_0\} \}$. 
Now, $G_{(x_0,x_0)}=(U \times U)_{(x_0,x_0)} \rtimes \langle \tau 
\rangle,$ and this gives that any $G_{(x_0,x_0)}$-orbit different 
from $\{ (x_0,x_0)\}$ is contained in either 
$\Delta_1 \cup \Delta_2$ or $\Delta_3$. Since the rank 
$r_G=3,$ we find that 
the nontrivial subdegrees of $G$ are 
$|\Delta_1 \cup \Delta_2|=2(q^{n/2}-1)$ and 
$|\Delta_3|=(q^{n/2}-1)^2$. On the other hand $|S_-|$ is a subdegree which is divisible by $q,$ see \eqref{S-} (we use here that $n\ge 3$).
\medskip

\noindent {\it Case 2.}  $k_2=3,$ $r_G=4$ and $G \le U \, wr \, S_3$. 
\medskip

In this case $\F_q^n$ can be written as $\F_q^n=X \times X \times X,$ $|X|=q^{n/3},$ and $U$  is a permutation group of $X$. The wreath product  $U \, wr \, S_3 = \langle U \times U \times U \rangle \rtimes K,$ where $U \times U \times U$ acts on $X \times X \times X$ naturally, 
$K\cong S_3,$ and $K$ acts by permuting the coordinates. 
The socle $H = T \times T \times T \le U \times U \times U$  and 
$T$ is $2$-transitive on $X$. 
Now, $G_{(x_0,x_0,x_0)} \le (U \times U \times U)_{(x_0,x_0,x_0)} \rtimes K,$ and this gives that any $G_{(x_0,x_0,x_0)}$-orbit different 
from $\{ (x_0,x_0,x_0)\}$ is contained in one of the sets 
$\{ (x,x_0,x_0), (x_0,x,x_0), (x_0,x_0,x) : x \in X \setminus 
\{ (x_0,x_0,x_0)\}\},$ 
$\{ (x,y,x_0), (x,x_0,y), (x_0,x,y) : x,y \in X \setminus \{ (x_0,x_0,x_0)\}\}$ and 
$\{ (x,y,z): x,y,z \in X \setminus \{ (x_0,x_0,x_0)\}\}$.
Because of this and $r_G=4$ we find that the nontrivial subdegrees of 
$G$ are  $3(q^{n/3}-1),\, 3(q^{n/3}-1)^2$ and $(q^{n/3}-1)^3$.
On the other hand these subdegrees are $|S_\varepsilon|,\,\varepsilon\in\{0,+,-\}$, and as $q^{\lceil\frac{n-2}{2}\rceil}$ 
divides both $|S_+|$ and $|S_-|$ and $hn$ is divisible by $3,$ 
we obtain that $(q,n)=(3,3),$ and therefore, $U\cong S_3$ 
and $T\cong \mathbb{Z}_3,$ contradicting that $T$ is nonabelian. 
\end{proof}

We are ready to settle (1). 

\begin{lem}\label{L3}
Let $\gamma\in \aut(\F_q^n)$ be an arbitrary automorphism 
and let $n\ge 3$. Then $\gamma$ satisfies \eqref{iff}. 
\end{lem}

\begin{proof}
Suppose for the moment that $q=p$. By Lemma~\ref{L2}, 
$E=\soc(G),$ in particular, $E$ is normal in $G$. Now, since $q=p,$
we obtain that $\gamma$ is an affine transformation, and this implies that  
it satisfies \eqref{iff}.

From now on it will be assumed that $q\ne p$. 
Assume to the contrary that there exist vectors $\va$ and $\vb$ 
such that either $d(\va,\vb)=0$ and $d(\va^\gamma,\vb^\gamma)\ne 0,$ or $d(\va,\vb) \ne 0$ and $d(\va^\gamma,\vb^\gamma)=0$.
Here we deal only with the first case because the second one  
can be treated in a very similar way.
Consider the product $\gamma':=\gamma_1\gamma\gamma_2$  where $\gamma_1$ and $\gamma_2$ are the translations 
$\vx \mapsto \vx+\va$ and $\vx\mapsto \vx-\va^\gamma,$ 
respectively. Then $\vz^{\gamma'}=\vz,$ $\vb-\va \in S_0,$ and 
$(\vb-\va)^{\gamma'}=\vb^\gamma-\va^\gamma \in S_+$.     
These imply that the $G_{\vz}$-orbits are $\{\vz\},S_0 \cup S_+$ and 
$S_-$ (see also the remark before Lemma~\ref{L2}), and thus 
$G$ has nontrivial subdegrees: 
\begin{equation}\label{subdeg}
|S_0|+|S_+| \text{ and } |S_-|.
\end{equation}

By Lemma~\ref{L2}, $G$ is of type (T1), that is, it is an 
affine permutation group. Therefore, we can use the classification of finite primitive affine permutation groups of rank $3$ due to Liebeck 
(see \cite[Theorem~1.1]{Lieb}). Namely, if $L$ is such a group acting an a vector space $V$ of cardinality $p^d,$ and $L_0$ 
denotes the stabilizer of the zero vector $0,$ then one of the following holds: 
\medskip

\noindent 
\emph{Infinite classes (A):} \ $L$ is in one of $11$ infinite classes of permutation groups labeled by (A1)--(A11).  
If $L$ is in class (A1), then $L_0$ is isomorphic to a subgroup of 
$\ggl(1,p^d);$ and if $L$ is in class (A2)--(A11), then $d=2r$ and 
$L$ has nontrivial subdegrees listed  in  Table~2 
(see~\cite[Table~12]{Lieb}).
\medskip

\begin{table}[t!]\label{table:L}
\begin{center} 
{\small 
\begin{tabular}{|c|c|c|} \hline
row     &  subdegrees                                             & conditions\\ \hline
        &                                                         & $s=0$ or \\
        &                                                         & $s \mid r$ or \\ 
1.      &  $(p^s+1)(p^r-1)$, \; $p^s(p^r-1)(p^{r-s}-1)$           & $s=2r/5$ and $5 \mid r$ or \\ 
        &                                                         & $s=3r/4$ and $4 \mid r$ or \\ 
        &                                                         & $s=3r/8$ and $8 \mid r$ \\   \hline       
2.      &  $(p^{r-s}+1)(p^r-1)$, \; $p^{r-s}(p^r-1)(p^s-1)$       &  $s \mid r$   \\  \hline
3.      &  $(p^{r-s}-1)(p^r+1)$, \; $p^{r-s}(p^r+1)(p^s-1)$       &  $s \mid r$ and $s \ne r$ \\  \hline
\end{tabular}
}
\end{center}
\caption{Nontrivial subdegrees of affine groups of rank $3$ in classes (A2)--(A11).}
\end{table}

\noindent \emph{`Extraspecial' classes (B):} \ $L$ is one of a finite set of permutation groups whose degree is equal to 
one of the following numbers (\cite[Table~1]{Lieb}): 
\begin{equation}\label{eq:L1}
2^6, \, 3^4,  \, 3^6, \, 3^8, \, 5^4, \,  7^2, \, 7^4, \, 13^2, \, 17^2, \, 19^2, \, 23^2, \, 29^2, \, 31^2, \, 47^2. 
\end{equation}

\noindent \emph{`Exceptional' classes (C):} \ $L$ is one of a finite set of permutation groups whose degree is equal to 
one of the following numbers (\cite[Table~2]{Lieb}): 
\begin{equation}\label{eq:L2}
2^6, \, 2^8, \, 2^{11}, \, 2^{12},  \, 3^4,  \,  3^5, \, 3^6, \, 3^{12}, \, 5^4, \,  5^6, \, 
7^4, \, 31^2, \, 41^2, \, 71^2, \, 79^2, \, 89^2. 
\end{equation}

We are going to arrive at a contradiction after comparing the 
subdegrees described in classes (A)--(C)
with our subdegrees in \eqref{subdeg}. 

Suppose that $G$ is in class (A). If $G$ is in class (A1), then $G_\vz$ is isomorphic to a subgroup of $\ggl(1,q^n),$ 
hence $|G_\vz|$ divides $|\ggl(1,q^n)| = hn(q^n-1)$. 
Each subdegree of $G$ divides $|G_\vz|$. In particular, 
$|S_-| \mid |G_\vz|,$ and by \eqref{S-}, 
$p^{h \lceil\frac{n-2}{2}\rceil} \mid hn(q^n-1)$. 
From this we obtain that $p^{m} \le 4m$ where $p$ is an odd prime and $m=h\lceil\frac{n-2}{2}\rceil \ge 2$ (recall that $n\ge 3$ and 
$h\ge 2$ because of $q\ne p$). This, however, contradicts the inequality $p^{m} > 4m,$ which can be easily settled by induction on $m$.

Let $G$ be in class (A$i$) for $i > 1$. As before, let 
$m=h \lceil\frac{n-2}{2}\rceil$. By \eqref{S-}, $p^m$ is the 
largest $p$-power dividing the subdegree $|S_-|,$ and we get 
$2|S_-|/p^m \equiv \pm 1 \pmod q$. Thus 
\begin{equation}\label{modp2}
2|S_-|/p^m \equiv \pm 1\, (\text{mod }p^2).
\end{equation}
Let us compute the residue of $2|S_-|/p^m$ modulo $p^2$ by the help of Table~1. Since $q^n=p^{2r},$  it follows that $2r=hn,$ and hence $r\ge 3$. 
Suppose that $|S_-|$ occurs in the 1st row of 
Table~1. In this case $m=s,$ and hence $s\ge 2$. 
Notice that, if $s \mid r,$ then $s \mid (r-s),$ from which  
$r-s \ge 2$ and this implies that $2|S_-|/p^m \equiv 2 \pmod {p^2},$ contradicting \eqref{modp2}. Now, it follows that $r-s \ge 2$ whenever  $r\ne 4$ and $s\ne 3$. Let $r=4$ and $s=3$. Then $hn=8,$ 
thus $m$ is even, which contradicts that $m=s=3$. 
Now, suppose that $|S_-|$ occurs in the 2nd or the 3rd row of 
Table~1. In this case $m=r-s,$ and if $s\ne 1,$ then 
$2|S_-|/p^m \equiv \pm 2 \pmod {p^2},$ contradicting 
\eqref{modp2}. Let $s=1$. Then $h\frac{n}{2}-1=r-1=m=h\lceil 
\frac{n-2}{2} \rceil$. We obtain that $h=2$ and $n$ is odd. 
Then $q=p^2 \equiv 1 \pmod 4$.  If $|S_-|$ is equal to a number in 
the 2nd row, then by \eqref{S-}, $p^{n+1}-p^{n-1}-p^2+1=2p^{n+1}-2p^n-2p+2,$ and if it is equal to a number in the 3rd row, then 
$p^{n+1}-p^{n-1}-p^2+1=2p^{n+1}-2p^n+2p-2$. 
It is easy to see that none of these equations holds for $n\ge 3$ and 
an odd prime $p$.

Suppose that the group $G$ is in class (B). 
We obtain from \eqref{eq:L1} that $(q,n)=(9,3)$ or  $(9,4)$.
The possibility that $q^n=9^3$ gives rise to a strongly regular graph on $9^3$ points with valency $|S_-|=288$. 
These numbers can be excluded by Brouwer's database of strongly regular graphs (see https://www.win.tue.nl/$\sim$aeb/graphs/srg/srgtab.html). As an alternative way to exclude this case, one may 
use the classification of solvable primitive permutation groups of rank $3$ due to Foulser \cite{F}. Let $q^n=9^4$. Then 
by \cite[Table~13]{Lieb}, the corresponding subdegrees are 
$1440$ and $5120$. 
However, these do not match the numbers given in \eqref{subdeg}.

Finally, suppose  that $G$ is in class (C). 
Then we obtain from \eqref{eq:L2} that $(q,n)\in \{(9,3),(25,3),(81,3),(27,4),(9,6)\}$.
By \cite[Table~14]{Lieb}, the corresponding nontrivial subdegrees are:
\begin{center}
\begin{tabular}{c|c}
$q^n$          & nontrivial subdegrees \\ \hline
$3^6$          &  $224$, $504$ \\ \hline
$5^6$          & $7560,$ $8064$ \\ 
$3^{12}$     &  $65520,$ $465920$ \\ \hline
\end{tabular}
\end{center}
However, none of these matches the numbers in \eqref{subdeg}. 
The lemma is proved.  
\end{proof}

\begin{rem}
We would like to note that in our earlier approach we gave a proof of   
Theorem~\ref{T-main}, which also relies on Lemmas~\ref{L1}-\ref{L3}, but instead of invoking    
Lester's result (Theorem~\ref{L}), we used the results of Iosevich et al.~\cite{ISX} on maximum point sets  
with any two of its points being at distance $0$.  Here we give an outline.
Let $\gamma \in \aut(\F_q^n)$ be an integral automorphism which fixes the zero vector. 
We need to prove that $\gamma$ is a semilinear transformation. By the fundamental theorem of projective geometry we 
are done if we show that  $\gamma$ preserves both the point and the line set of the projective space $\pg(n-1,q)$. 
Let us consider the nonsingular quadric $\Q$ of $\pg(n-1,q)$ induced by the quadratic form 
$x_1^2+\cdots+x_n^2$.  A projective subspace of maximum dimension on $\Q$ is called a \emph{generator} 
(cf.~\cite[Chapter~22]{HT}).  Observe that any subspace $U$ of $\F_q^n$ corresponding to a generator has the property 
that any two of its points are at distance $0$. It follows from \cite[Theorem~2 and Lemma~4]{ISX} that 
$U$ is a maximum point set with the latter property, and thus $\gamma$ maps $U$ to a subspace. The latter subspace is contained 
in $S_0,$ see Lemma~\ref{L3}, and we conclude that $\gamma$ permutes the generators among themselves.
This observation and the fact that any point of $\Q$ can be expressed as the intersection of some generators yield that 
$\gamma$ preserves the set of points on $\Q$. Then, using Lemma~\ref{L2}, we find that any line of $\pg(n-1,q)$ 
through two points of $\Q$ is mapped by $\gamma$ to a line.  If $(n,q) \ne (3,3),$ then any point of $\pg(n-1,q)$ can be expressed 
as the intersection of two lines each connecting two points of 
$\Q,$ and this with the previous observation yield that $\gamma$ preserves the point set of $\pg(n-1,q)$. Then, using again Lemma~\ref{L2}, we conclude that $\gamma$ preserves 
the line set of $\pg(n-1,q)$ as well.

Finally, we would like to mention that 
Theorem~\ref{L} was rediscovered by Vroegindewey 
\cite{V82} (the special case when the associated quadratic 
form has Witt index $1$ was treated already in \cite{V74}).
\end{rem}

\section*{Acknowledgements}
The authors are grateful to Marko Orel for drawing their attention 
to the work of Lester \cite{Les}.
This research was supported in part by the OTKA-ARRS Slovenian-Hungarian Joint Research Project, grant no.\ NN 114614 (in Hungary) and 
N1-0032 (in Slovenia).  The first two authors also thank the Slovenian Research Agency ARRS (research program P1-0285 and research projects  N1-0038, J1-5433, J1-6720 and J1-6743), and the second author was also supported in part by 
WoodWisdom-Net+, W$^3$B.

\end{document}